  \documentclass[12pt,a4paper]{article}
  \usepackage{amsfonts,fullrefs,std-defs,microtype}

\newcommand\chapart{paper} % This is not a chapter, but a standalone paper.

\begin{document}

\title{\bf Asymptotics and scalings for large product-form networks via
       the Central Limit Theorem}
\author{Guy Fayolle 
       \thanks{Postal address: INRIA --- 
                     Domaine de Voluceau, Rocquencourt, BP105 --- 
                     78153 Le Chesnay, France.}
       \and Jean-Marc Lasgouttes$\ ^*$}

\date{February 1996}

\maketitle

\begin{abstract}
The asymptotic behaviour of a closed BCMP network, with $n$ queues and
$m_n$ clients, is analyzed when $n$ and $m_n$ become simultaneously
large. Our method relies on Berry-Esseen type approximations coming in
the Central Limit Theorem. We construct {\em critical sequences\/}
$m^{\scriptscriptstyle0}_n$, which are necessary and sufficient to
distinguish between saturated and non-saturated regimes for the
network.  Several applications of these results are presented.  It is
shown that some queues can act as bottlenecks, limiting thus the
global efficiency of the system.
\end{abstract}
\nocite{FayLas:3}
% Fichier principal pour BigNet

% Nicer...
\let\oldphi=\phi
\let\phi=\varphi
\let\varphi=\oldphi
\let\eps=\varepsilon
\renewcommand\liminf{\mathop{\rm \underline{lim}}}
\renewcommand\limsup{\mathop{\rm \overline{lim}}}
\def\Re{\mathop{\rm Re}\nolimits}
\def\Im{\mathop{\rm Im}\nolimits}
\def\ci{{\rm i}} % complex number i

% Notation
\def\*#1{{\cal #1}}
\def\m#1{\beta^{\scriptscriptstyle(#1)}}
\def\0{^{\scriptscriptstyle0}}
\def\wh{\widehat}
\def\wt{\widetilde}

% Little hyphenation problem
\hyphenation{uni-form-ly}
            
\section{Introduction}

In many applications (telecommunications, transportation, etc.), it is
desirable to understand the behaviour and performance of stochastic
networks as their size increases. From an engineering point of view,
the problem can be roughly formulated as follows:

\begin{quote}
\em Consider a closed network with $n$ nodes and exactly $m_n$ customers
circulating inside. Find a function $f$, such that $m=f(n)$ yields
an {\em interesting\/} performance of the system as $n$ increases.
\end{quote}

In this study, we start from the so-called {\em product-form\/}
networks, which play an important role in quantitative analysis of
systems. Although the equilibrium state probabilities have then a
simple expression (see for example Kelly~\cite{Kel:1}), non-trivial
problems remain, due to an intrinsic combinatorial explosion in the
formulas, especially in those involving the famous normalizing
constant. To circumvent these drawbacks, the idea is to compute
asymptotic expansions of the characteristic values of the network,
when $m$ and $n$ both tend to infinity.

This approach has been considered by Knessl and Tier~\cite{KneTie:1},
Kogan and Birman~\cite{Kog:1,KogBir:1,BirKog:1} and Malyshev and
Yakovlev~\cite{MalYak:1}. However, it relies on purely analytical
tools, which are difficult to use in a more general setting and, in
our opinion, do not really give a structural explanation of the
phenomena involved.

The method proposed hereafter has direct connections with the {\em
Central Limit Theorem\/}: instead of representing the values of
interest as complex integrals, we express them in terms of
distributions of scaled sums of independent random variables. Besides
giving a clear interpretation of the computations, this allows to
handle directly the general case of single-chain closed networks.  We
show by construction the existence of {\em critical sequences\/}
$m\0_n$ in the following sense: the network saturates if, and only if,
$m_n\gg m\0_n$.  These results can also be interpreted as {\em
insensitivity\/} properties: as the number of stations $n$ and the
number of customers $m_n$ go to infinity, the network is shown to be
equivalent to an {\em open\/} network of $n$ independent queues (having
a total mean number of customers $m_n$), in the sense that both
systems have asymptotically the same finite-dimensional distributions.

The \chapart\ is organized as follows. The model is introduced in
Section~\ref{sec:model}, together with a presentation of the method.
In Section~\ref{sec:llt}, asymptotics of the marginal distribution of
the queue lengths are given under normal conditions and also when some
queues become overloaded.  Section~\ref{sec:scal} unifies the results
and contains the main theorems about scaling. Section~\ref{sec:ass}
and \ref{sec:appl} are devoted to concrete applications of these
results, in particular to service vehicle networks (like the {\em
Praxi\-t\`ele\/} project, now developed at INRIA).
Section~\ref{sec:concl} contains some conclusive remarks. Most of the
technical proofs are postponed in Appendix.

\section{Mathematical model and view of the main results}
\label{sec:model}

Consider a closed BCMP network $\*C_n$ with $n$ queues and $m_n$
clients. The number of clients at queue $k$ at steady state is a
random variable $Q_{k,n}$. The service rate at queue $k$ when there
are $q_k$ customers is $\mu_{k,n}(q_k)$. The routing probability from
queue $k$ to queue $\ell$ is $p_{k,\ell,n}$ and $P_n$ denotes the
transition matrix supposed to be ergodic, with invariant measure
$\pi_n=(\pi_{1,n},\ldots,\pi_{n,n})$, defined by:
\begin{equation}\label{eq:invmeas}
 \pi_nP_n=\pi_n\mbox{\rm \ and }\pi_{1,n}+\cdots+\pi_{n,n}=1.
\end{equation}

Then it is known that, for any $q_1,\ldots,q_n\geq0$ such that
$q_1+\cdots+q_n=m_n$,
\begin{equation}\label{eq:prodform}
 \PP_n(Q_{1,n}=q_1,\ldots,Q_{n,n}=q_n) = 
   Z_{m_n,n}^{-1}\prod_{k=1}^n 
     {\pi_{k,n}^{q_k}\over \mu_{k,n}(1)\cdots \mu_{k,n}(q_k)},
\end{equation}
with the normalizing condition
\begin{equation}\label{eq:prodform:const}
 Z_{m,n}=\sum_{q_1+\cdots+q_n=m}\prod_{k=1}^n 
  {\pi_{k,n}^{q_k}\over \mu_{k,n}(1)\cdots \mu_{k,n}(q_k)}.
\end{equation}

It is worth noting that our analysis applies to any network which has
a product form equilibrium distribution like~(\ref{eq:prodform}). It
includes for example, as soon as the matrix $P_n$ is reversible, all
systems having transition rates of the form
$p_{k,\ell,n}\alpha_{k,n}(q_k)\beta_{\ell,n}(q_\ell)$, in which case
finite capacity situations can be covered
(e.g. $\beta_{\ell,n}(q_\ell)=\ind{q_\ell\leq\bar q_\ell}$). See
Serfozo~\cite{Serf:2} for further examples.

To avoid hiding global results with tedious technicalities, we suppose
throughout the study that, for all $n$, $\*C_n$ contains at least one
queue which, taken in isolation, can be saturated with a finite input
flow (e.g.\ a $M/M/c/\infty$ queue). 

The overall presentation requires definitions and an intermediate lemma,
given in Section~\ref{ssec:model:def}. The informal presentation of
the central results appears in Section~\ref{ssec:model:method}.

\subsection{Preliminaries}\label{ssec:model:def}

Define, for each $k$, $1\leq k\leq n$, the generating function
\[
 f_{k,n}(z) \egaldef \sum_{q=0}^\infty 
   {z^q\over \mu_{k,n}(1)\cdots \mu_{k,n}(q)}.
\]

Note that for each $n$, $f_{k,n}$ has a singularity at finite
distance for at least one $1\leq k\leq n$.

To the original closed network $\*C_n$, we let correspond a new system
$\*O_n(\lambda)$ which is open and consists of $n$ parallel queues,
with service rates $\mu_{k,n}(x)$ and arrival intensity
$\lambda\pi_{k,n}$ at queue $k$, where the choice of $\lambda$ will be
made more precise later. The queue length $X_{k,n}(\lambda)$ of the
$k$-th queue of $\*O_n(\lambda)$ has a distribution given by
\[
 P(X_{k,n}(\lambda)=x)={1\over f_{k,n}(\lambda\pi_{k,n})}
              {(\lambda\pi_{k,n})^x
                 \over \mu_{k,n}(1)\cdots\mu_{k,n}(x)},
\]
and $X_{1,n}(\lambda),\ldots,X_{n,n}(\lambda)$ are independent
variables.  We assume that $X_{k,n}(\lambda)$ has some finite moments
of order $r\geq 2$ and introduce the following notation:

\[
\begin{array}{rclcrcl}
m_{k,n}(\lambda)        &\egaldef& \EE X_{k,n}(\lambda),&&
S_n(\lambda)            &\egaldef& \sum_{k=1}^nX_{k,n}(\lambda),\\[0.3ex]
\m{r}_{k,n}(\lambda)    &\egaldef& \EE |X_{k,n}(\lambda)-m_{k,n}(\lambda)|^r,&&
\m{r}_n(\lambda)        &\egaldef& \sum_{k=1}^n\m{r}_{k,n}(\lambda),\\[0.3ex]
\sigma^2_{k,n}(\lambda) &\egaldef& \m2_{k,n}(\lambda),&&
\sigma_n^2(\lambda)     &\egaldef& \m2_n(\lambda),\\[0.3ex]
\bar\m3_{k,n}(\lambda)  &\egaldef& \EE [X_{k,n}(\lambda)-m_{k,n}(\lambda)]^3,&&
\bar\m3_n (\lambda)     &\egaldef& \sum_{k=1}^n\bar\m3_{k,n}(\lambda).
\end{array}
\]

Let $\phi_{k,n}(\theta\,;\lambda)$ be the characteristic
function of $X_{k,n}(\lambda)-m_{k,n}(\lambda)$. Then, for any real $\theta$, 
\begin{equation}\label{eq:defphikn}
 \phi_{k,n}(\theta\,;\lambda)
   \egaldef\EE e^{\ci(X_{k,n}(\lambda)-m_{k,n}(\lambda))\theta}
   = {f_{k,n}(\pi_{k,n}\lambda e^{\ci\theta})
         \over f_{k,n}(\pi_{k,n}\lambda)}e^{-\ci m_{k,n}(\lambda)\theta},
\end{equation}
and 
\begin{equation}\label{eq:defphin}
 \phi_n(\theta\,;\lambda)
   \egaldef\EE e^{\ci(S_n(\lambda)-\EE S_n(\lambda))\theta}
   = \phi_{1,n}(\theta\,;\lambda)\cdots\phi_{n,n}(\theta\,;\lambda)
\end{equation}

The reason why $\*O_n(\lambda)$ has been introduced is that the main
performance characteristics of the network $\*C_n$ can be expressed
simply in terms of the distribution of
$X_{1,n}(\lambda),\ldots,X_{n,n}(\lambda)$: 

\begin{lem}\label{lem:represent}
\begin{enumerate}
\item
For any choice of $m_n$, there exists a unique $\lambda_n$ such
that 
\begin{equation}\label{eq:lambdan}
  \EE S_n(\lambda_n) 
   = \EE[X_{1,n}(\lambda_n)+\cdots+X_{n,n}(\lambda_n)] = m_n.
\end{equation}

 From now on, unless otherwise stated, all quantities will pertain to
the network $\*O_n(\lambda_n)$ and $\lambda_n$ will be omitted.

\item
Equations~(\ref{eq:prodform}) and~(\ref{eq:prodform:const}) can be
rewritten as
\begin{equation}\label{eq:pf2}
\PP(Q_{1,n}=q_1,\ldots,Q_{n,n}=q_n)
 = {1\over \PP(S_n=m_n)}\prod_{k=1}^n \PP(X_{k,n}=q_k).
\end{equation}

\item
For any $\ell>0$ and $q_1,\ldots,q_\ell\geq0$, the joint distribution
of the number of customers in the queues $1,\ldots,\ell$ of $\*C_n$ is
\begin{eqnarray}\label{eq:marg}
\lefteqn{\PP(Q_{1,n}=q_1,\ldots,Q_{\ell,n}=q_\ell)}\qquad\quad\\
 &=& {\PP(S_n-\sum_{k=1}^\ell X_{k,n} = m_n-\sum_{k=1}^\ell q_k)
      \over \PP(S_n=m_n)}\prod_{k=1}^\ell \PP(X_{k,n}=q_k)\nonumber\\
 &=& \PP(X_{1,n}=q_1,\ldots,X_{\ell,n}=q_\ell\big|S_n=m_n),\nonumber
\end{eqnarray}
and, consequently, $\EE Q_{\ell,n} = \EE[X_{\ell,n}\big|S_n=m_n]$.

\item For any $1\leq \ell\leq n$, 
\begin{equation}\label{eq:mean}
\EE Q_{\ell,n} = m_{\ell,n}{\PP(S_n-X_\ell+\wt X_\ell=m_n)
                       \over \PP(S_n=m_n)},
\end{equation}
where $\wt X_{\ell,n}$ is an integer-valued r.v., independent from
everything else and having distribution
\[
 \PP(\wt X_{\ell,n}=x)={x\PP(X_{\ell,n}=x)\over m_{\ell,n}}.
\]
\end{enumerate}
\end{lem}

Note that $\lambda_n$ can be obtained as the unique solution of the
equation
\begin{equation}\label{eq:eqlambdan}
 m_n=\sum_{k=1}^n {\pi_{k,n}\lambda_n f'_{k,n}(\pi_{k,n}\lambda_n)
                    \over f_{k,n}(\pi_{k,n}\lambda_n)}.
\end{equation}

While this equation is in general impossible to solve explicitly,
$\lambda_n$ can be computed numerically using classical
methods.

\begin{proof}{}
A straightforward computation yields, for all $1\leq k\leq n$,
\begin{equation}\label{eq:derm}
{\partial m_{k,n}(\lambda)\over\partial\lambda}
  = {\sigma_{k,n}^2(\lambda)\over\lambda}\;>\;0.
\end{equation}

The mean number of clients in $\*O_n(\lambda)$ is thus a strictly increasing
function of $\lambda$, which equals zero when $\lambda=0$ and goes
to infinity with $\lambda$. This proves the first assertion of the
lemma.

Define 
\[
 Y_n\egaldef {Z_{m_n,n}\lambda_n^{m_n}
                \over \prod_{k=1}^n f_{k,n}(\pi_{k,n}\lambda_n)}.
\]

Then (\ref{eq:prodform}) reads
\[
 \PP_n(Q_{1,n}=q_1,\ldots,Q_{n,n}=q_n) = 
   {1\over Y_n}\prod_{k=1}^n \PP(X_{k,n}=q_k),
\]
which yields (\ref{eq:pf2}), since
\[
 Y_n=\sum_{q_1+\cdots+q_n=m_n}\prod_{k=1}^n \PP(X_{k,n}=q_k)
    =\PP(X_{1,n}+\cdots+X_{n,n}=m_n).
\]

Equation~(\ref{eq:mean}) and the first part of (\ref{eq:marg}) are
derived similarly. For the second part of (\ref{eq:marg}), we simply
note that
\begin{eqnarray*}
\lefteqn{{\textstyle\PP(S_n-\sum_{k=1}^\ell X_{k,n} 
                       = m_n-\sum_{k=1}^\ell q_k)}
    \prod_{k=1}^\ell \PP(X_{k,n}=q_k)}\qqqq\\
 &=& \PP(S_n=m_n|X_{1,n}=q_1,\ldots,X_{\ell,n}=q_\ell)
     \prod_{k=1}^\ell \PP(X_{k,n}=q_k)\\
 &=& \PP(X_{1,n}=q_1,\ldots,X_{\ell,n}=q_\ell|S_n=m_n)
     \PP(S_n=m_n).
\end{eqnarray*}
\end{proof}

\subsection{Informal description of the method}\label{ssec:model:method}

Most of the derivations obtained in the \chapart\ are based on the
various representations given in Lemma~\ref{lem:represent}. Whereas
the studies \cite{Kog:1,KogBir:1,BirKog:1,MalYak:1} use mainly
saddle-point methods, our approach relies on direct limit theorems for
the distribution of $S_n$.

For example, assume that $S_n-m_n$ satisfies a local limit theorem
such as:

\begin{quote}\em 
Under ``suitable'' conditions, there exists a distribution with density
$f$ and a sequence $a_n$ such that, for any integer $x$,
\begin{equation}\label{eq:model:llt}
 \lim_{n\to\infty} a_n\PP(S_n-m_n=x) - f\Bigl({x\over a_n}\Bigr) = 0.
\end{equation}
\end{quote}

Then Lemma~\ref{lem:represent} will yield
\[
\PP(Q_{1,n}=q_1,\ldots,Q_{n,n}=q_n)
 \;\approx\; {a_n\over f(0)}\prod_{k=1}^n \PP(X_{k,n}=q_k),
\]
and, for any finite $\ell$, 
\[
\PP(Q_{1,n}=q_1,\ldots,Q_{\ell,n}=q_\ell)
 \;\approx\; \prod_{k=1}^\ell \PP(X_{k,n}=q_k).
\]

This amounts to say that the joint distribution of any {\em finite}
number of queues in the BCMP network $\*C_n$ is, at steady state,
asymptotically equivalent to the product distribution of the
corresponding queues in the system $\*O_n$.

It is at this moment important to emphasize that we {\em do not}
require any ``smooth'' limiting behaviour for $\*O_n$, which is
somehow an instrumental network, computationally easier to evaluate.

\medskip
To prove local limit theorems like~(\ref{eq:model:llt}), it is
necessary to investigate carefully the behaviour of the variables
$X_{k,n}$.  In particular, since $\EE S_n=m_n<\infty$, all queues in
$\*O_n$ are ergodic, which reads, for any $1\leq k\leq n$,
\[
 \lambda_n\pi_{k,n}<\mu_{k,n}\leq\infty, 
\]
or, equivalently,
\begin{equation}\label{eq:defrhon}
 \rho\0_n\egaldef\lambda_n\max_{1\leq k\leq n}{\pi_{k,n}\over\mu_{k,n}}<1,
\end{equation}
where typically 
\[ 
 \mu_{k,n}=\liminf_{q\to\infty}\sqrt[q]{\mu_{k,n}(1)\cdots\mu_{k,n}(q)}.
\]

Three main situations have been analyzed:
\begin{enumerate}
\item {\em $\rho\0_n$ is bounded away from $1$\/}: then $S_n/\sigma_n$
satisfies a local Central Limit Theorem and tends in distribution to a
{\em normal} law (see Theorem~\ref{thm:m0n});
\item {\em $\rho\0_n\to1$ and the supremum in~(\ref{eq:defrhon}) is
attained for a finite number of queues\/}: then the network subdivides
into two subsets, the ``saturated'' queues and the rest of the
network. As shown in Theorem~\ref{thm:scal}, under mild regularity
assumptions, there exists a sequence $\alpha_n$ such that 
$S_n/\alpha_n$ tends to a {\em gamma} law;
\item {\em $\rho\0_n\to1$ and the supremum in~(\ref{eq:defrhon}) is
attained for an unbounded number of queues\/}: $S_n/\sigma_n$ again
tends in distribution to a {\em normal} law (see
Theorem~\ref{thm:xiinf}).
\end{enumerate}

In fact, Theorems~\ref{thm:m0n}, \ref{thm:scal} and \ref{thm:xiinf}
quoted above are general, in the sense that they provide a
construction of efficient scalings in terms of $m_n$, the number of
customers: the existence of {\em critical sequences} $m\0_n$ for the
network $\*C_n$ is shown by explicit construction. Under reasonable
assumptions, these sequences are {\em necessary and sufficient} to
discriminate between saturated and non saturated regimes. This is
similar to phase transition phenomena observed in~\cite{MalYak:1},
where it was assumed that $m_n/n\to\lambda>0$ (see
Section~\ref{ssec:MalYak}). Clearly, for a non-saturated regime to
exist as $n\to\infty$, it is necessary to have $m_n=O(n)$; this
condition is not sufficient (see Section~\ref{ssec:tight}).

Condition~(\ref{eq:defrhon}) can be used to determine an upper bound
for $\lambda_n$ and to exhibit queues which act as bottlenecks in the
network $\*C_n$ (see Section~\ref{sec:scal}). 

\medskip
{\noindent\bf Remark }Rather than simple limit theorems, the results in
Sections~\ref{sec:llt} and~\ref{sec:scal} are given in terms of
asymptotic expansions, using the operators $O$ and $\Omega$ defined as
follows:
\begin{eqnarray*}
 a(\eta) = O(b(\eta)), 
  &\mbox{\rm iff} &
   \exists K>0,\ \forall \eta,\ |a(\eta)|\leq K|b(\eta)|,\\
 a(\eta) = \Omega(b(\eta)), 
  &\mbox{\rm iff} & 
    a(\eta) = O(b(\eta)) \mbox{\rm\ and } b(\eta)=O(a(\eta)),
\end{eqnarray*}
where $\eta$ is some unspecified argument. Unless otherwise stated, all
these bounds are {\em uniform} with respect to $n$ and all queue indexes.

\section{Local limit theorems and asymptotic expansions}
\label{sec:llt}

In this section, we compute estimates of several performance measures
of $\*C_n$ by means of local limit theorems on sums of independent
random variables. The two series of results presented here are of
somewhat different nature: whereas the conditions of
Proposition~\ref{pro:clt:norm} depend on moments,
Proposition~\ref{pro:lt:pole} relies on analytic properties of the
generating function of some queues.

\subsection{Normal traffic case}\label{ssec:stab}

When the queues are not saturated (in a sense made more precise in
Proposition~\ref{pro:clt:norm}), it is possible to prove local Central
Limit Theorems, relying more exactly on Berry-Esseen type expansions
(see for instance Feller~\cite{Fel:1}).

Define $\gamma^2_{k,n}$ from $X_{k,n}$ as in Lemma~\ref{lem:majcar} of
the appendix, and let 
\[ 
 \gamma^2_n \egaldef \gamma^2_{1,n}+\cdots+\gamma^2_{n,n} \leq \sigma_n^2.  
\]

\begin{pro}\label{pro:clt:norm}
\begin{enumerate}
\item Let, for any $0<r\leq1$ such that $\m{2+r}_n$ exists,
\[
 \delta_n^r\egaldef {1\over2}{\sigma^2_n\over\m{2+r}_n}.
\]

Let $\gamma_n\delta_n\to\infty$ as $n\to\infty$. Then, for any integer
$x$, the following approximation holds {\em uniformly in $x$}:
\begin{eqnarray}\label{eq:clt:norm}
\lefteqn{\sigma_n\PP(S_n-m_n=x)
    -{1\over\sqrt{2\pi}}e^{-{x^2\over2\sigma_n^2}}}\qqqq\qqqq\\
   &=& O\biggl({\m{2+r}_n\over\sigma_n^{2+r}}\biggr)
      +O\biggl({\sigma_n\over\gamma_n^2\delta_n}
           \exp\Bigl(-{\gamma^2_n\delta_n^2\over5}\Bigr)\biggr).\nonumber
\end{eqnarray}
\item Let, for any $0<r\leq1$ such that $\m{3+r}_n$ exists,
\[
 \delta_n\egaldef {1\over2}{\sigma^2_n\over\m3_n}.
\]

Let $\gamma_n\delta_n\to\infty$ as $n\to\infty$. Then, for any integer
$x$, the following approximation holds {\em uniformly in $x$}:
\begin{eqnarray}\label{eq:clt:norm2}
\lefteqn{\sigma_n\PP(S_n-m_n=x)
    -{1\over\sqrt{2\pi}}e^{-{x^2\over2\sigma_n^2}}
    \biggr[1+{\bar\m3_n\over6\sigma_n^3}
               \Bigl({x^3\over\sigma_n^3}-3{x\over\sigma_n}\Bigr)
    \biggr]}\qqqq\qqqq\\
   &=& O\biggl({\m{3+r}_n\over\sigma_n^{3+r}}\biggr)
      +O\biggl({\sigma_n\over\gamma_n^2\delta_n}
           \exp\Bigl(-{\gamma^2_n\delta_n^2\over5}\Bigr)\biggr).\nonumber
\end{eqnarray}
\end{enumerate}
\end{pro}

\begin{proof}{}
  See Appendix~\ref{app:proof}
\end{proof}

The main assumption of the previous proposition is classical, since it is
nothing else but Lyapounov's condition, popular in the Central Limit
Problem:
\begin{equation}\label{eq:lyap}
 \mbox{\em for some }r>0,
        \ \lim_{n\to\infty} {\m{2+r}_n\over\sigma_n^{2+r}}=0.
\end{equation}

This condition yields in particular (see e.g.\ Lo\`eve~\cite{Loe:1})
\begin{equation}\label{eq:uan}
 \lim_{n\to\infty} \max_{1\leq k\leq n} {\sigma_{k,n}\over\sigma_n} =0,
\end{equation}
which in turn implies the {\em uniform asymptotic negligibility\/} of
the $X_{k,n}$'s. Note that it would be possible by truncation methods
to prove similar results without requiring the existence of moments.

We are now in a position to present some basic estimates when the size
of the network increases.

\begin{thm}\label{thm:stab}
Let $r$ be a real number such that $0<r\leq 2$.  Assume that
$\sigma_n=O(\gamma_n)$, that $\m{2+r}_n$ exists and
$\m{2+r}_n/\sigma_n^{2+r}\to0$ as $n\to\infty$. Then the following
asymptotic expansions hold.
\begin{enumerate}
\item 
\begin{eqnarray}\label{eq:normal:distrib}
\lefteqn{\PP(Q_{1,n}=q_1,\ldots,Q_{n,n}=q_n)}\qqqq\\
 &=& \sqrt{2\pi}\sigma_n \prod_{k=1}^n\PP(X_{k,n}=q_k)
   \biggl[1+O\biggl({\m{2+r}_n\over \sigma_n^{2+r}}\biggr)\biggr].\nonumber
\end{eqnarray}
\item For any finite $\ell$, if 
$[\sum_{j=1}^\ell m_{j,n}-q_j]/\sigma_n\to0$,
\begin{equation}\label{eq:normal:marg}
\PP(Q_{1,n}=q_1,\ldots,Q_{\ell,n}=q_\ell)=
 \prod_{k=1}^\ell\PP(X_{k,n}=q_k)\Bigl[1+O(\eps_{1,n})\Bigr],
\end{equation}
\begin{eqnarray*}
\eps_{1,n}
  &=& {\m{2+r}_n\over \sigma_n^{2+r}}
        + {\sum_{j=1}^\ell\sigma_{j,n}^2
             +(\sum_{j=1}^\ell m_{j,n}-q_j)^2\over\sigma_n^2}\\
  & & {}+\ind{r>1}{(\sum_{j=1}^\ell m_{j,n}-q_j)\bar\m3_n\over\sigma_n^4}.
\end{eqnarray*}
\item For any $j$,
\begin{equation}\label{eq:normal:mean}
 \EE Q_{j,n} = \EE X_{j,n}\Bigl[1+O(\eps_{2,n})\Bigr],
\end{equation}
\[
\eps_{2,n} = {\m{2+r}_n\over \sigma_n^{2+r}}
             +{\sigma_{j,n}^2\over\sigma_n^2}
       +{\m{2+r}_{j,n}\over m_{j,n}\sigma_n^{1+r}}
       +\ind{r>1}{\bar\m3_n\over\sigma_n^4}\sigma_{j,n}
            \Bigl(1+{\sigma_{j,n}\over m_{j,n}}\Bigr).
\]
\end{enumerate}
\end{thm}

\begin{proof}{}
Equation~(\ref{eq:normal:distrib}) is a simple application of
Proposition~\ref{pro:clt:norm} to (\ref{eq:pf2}).

To prove (\ref{eq:normal:marg}) from (\ref{eq:marg}) when $r\leq1$, we
simply write
\begin{eqnarray*}
 \lefteqn{{\PP(S_n-\sum_{k=1}^\ell X_{k,n} 
                       = m_n-\sum_{k=1}^\ell q_k)
           \over \PP(S_n=m_n)}}\qquad\quad\\
 &=& \Bigl(1-{\textstyle{\sum_{j=1}^\ell\sigma_{j,n}^2
             \over\sigma_n^2}}\Bigr)^{-{1\over2}}
   \biggl[1+O\biggl({\m{2+r}_n\over \sigma_n^{2+r}}
           +e^{-{\bigl(\sum_{j=1}^\ell m_{j,n}-q_j\bigr)^2
                        \over2\sigma_n^2}}
           -1\biggr)\biggr],
\end{eqnarray*}
and use the relation $|e^{-u^2/2}-1|\leq u^2$. When $1<r\leq 2$, it
suffices to take into account the inequality
\[
 (u^3-3u)e^{-u^2/2}=O(u).
\]

Relation~(\ref{eq:normal:mean}) is also derived from (\ref{eq:marg}).
\end{proof}

\subsection{Heavy traffic case}\label{ssec:pole}

We proceed now to analyze the behavior of the network $\*C_n$ when
some queues saturate, as $n\to\infty$. This, in particular, implies
that the Lyapounov condition~(\ref{eq:lyap}) is no more valid. In
fact, after a suitable normalization, $S_n-m_n$ will be shown to
converge in distribution to a random variable having a gamma
distribution, under the broad assumption that the first singularities
of the relevant generating functions are algebraic.

Let, for some $\rho\0_n\in[0,1[$ %] for emacs...
(to be specified in Section~\ref{sec:scal}),
\[
 \omega_n(\theta)\egaldef{1-\rho\0_n\over 1-\rho\0_ne^{\ci\theta}},
\]
and assume

\begin{ass}\label{ass:pole}
There exists a set $\*F\0_n$ of ``saturable'' queues, such that, for
all $k\in\*F\0_n$, there exist a real number $\xi_{k,n}$ and a
function $\psi_{k,n}(\theta)$ satisfying the relation
\[
 \phi_{k,n}(\theta)=e^{-\ci m_{k,n}\theta}
                    \omega_n^{\xi_{k,n}}(\theta)\psi_{k,n}(\theta).
\]

Moreover, $\psi'_{k,n}(\theta)=O(1)$, uniformly in $k$ and $n$, and
there exists a constant $\xi_{\max}$ such that
\[
 1\leq\xi_{k,n}<\xi_{\max}<\infty.
\]
\end{ass}

Clearly, the term $\omega_n^{\xi_{k,n}}(\theta)$ coming in the
definition of $\phi_{k,n}(\theta)$ emphasizes the fact that the
generating function $f_{k,n}(z)$ pertaining to queue $k\in\*F\0_n$ has
its first singularity which is algebraic of order $\xi_{k,n}$. If, in
addition, $\rho\0_n\to1$ as $n\to\infty$, the working conditions of
the system ensure all queues in $\*F\0_n$ saturate so that, in
particular, $\EE X_{k,n}\sim\xi_{k,n}\alpha_n$, where
\[
\alpha_n\egaldef{\rho\0_n\over1-\rho\0_n}.
\]

While this assumption covers a wide range of known queues, it is clear
that other types of singularities could be handled via the same
method.

Let
\[
 \xi_n\egaldef\sum_{k\in\*F\0_n}\xi_{k,n}.
\]
and define the total characteristic function of the queues in
$\*F_n\setminus\*F\0_n$ by
\[
 \wh\phi_n(\theta)\egaldef \prod_{k\not\in\*F\0_n}\phi_{k,n}(\theta).
\]

Let $r$ be a real number, $0<r\leq1$. Hereafter, $\hat\sigma_n$,
$\hat\m{2+r}_n$, $\hat\gamma_n$ and $\hat\delta_n$ will denote
quantities having the same meaning as in
Proposition~\ref{pro:clt:norm}, but related to $\wh\phi_n(\theta)$.

The counterpart of Theorem~\ref{thm:stab} now reads, in the case of
heavy operating conditions:
\begin{pro}\label{pro:lt:pole}
Let $\rho_n\0\to1$. If $\xi_n$ is bounded, $\hat\sigma_n/\alpha_n\to0$
and $\hat\delta_n\hat\gamma_n\to\infty$ as $n\to\infty$, then the
following estimate holds:
\begin{eqnarray}\label{eq:pro:lt:pole}
\lefteqn{\alpha_n\PP(S_n-m_n=x)
    -{(\xi_n+{x\over\alpha_n})^{\xi_n-1}e^{-\xi_n+{x\over\alpha_n}}
                               \over\Gamma(\xi_n)}}\qqqq\\
 & = &O\biggl(\Bigl({\hat\sigma_n\over\alpha_n}\Bigr)^2+
              {1\over\alpha_n}+{\hat\m{2+r}_n\over\hat\sigma_n^{2+r}}
                  \Bigl({\hat\sigma_n\over\alpha_n}\Bigr)^{2+r}
                   +{\hat\m{2+r}_n\over\hat\sigma_n^{2+r}}
                  \Bigl({\hat\sigma_n\over\alpha_n}\Bigr)^{\xi_n-1}
                   \biggr)\nonumber\\
 &   & {}+O\biggl({e^{-{\hat\gamma_n^2\hat\delta_n^2\over5}}
               \over\hat\gamma_n^2\hat\delta_n^{\xi_n+1}\alpha_n^{\xi_n-1}}
        \biggr).\nonumber
\end{eqnarray}
\end{pro}

\begin{proof}{}
  See Appendix~\ref{app:proof}
\end{proof}

The estimates of Proposition~\ref{pro:lt:pole} allow to establish the main
result of this section.

\begin{thm}\label{thm:pole}
Let $\hat\sigma_n/\alpha_n\to0$,
$\hat\m{2+r}_n/\hat\sigma^{2+r}_n\to0$ and
$\hat\sigma_n=O(\hat\gamma_n)$ as $n\to\infty$. Then the following
expansions hold when $\xi_n$ is uniformly bounded:
\begin{enumerate}
\item for any $q_1,\ldots,q_n\geq0$,
\begin{eqnarray}\label{eq:gamma:prob}
\lefteqn{\PP(Q_{1,n}=q_1,\ldots,Q_{n,n}=q_n)}\qqqq\\
  &=&{\alpha_n\Gamma(\xi_n)\over e^{-\xi_n}\xi_n^{\xi_n-1}}
   \prod_{k=1}^n\PP(X_{k,n}=q_k)\Bigl[1+O(\eps_n)\Bigr],\nonumber
\end{eqnarray}
with
\[
 \eps_n=\Bigl({\hat\sigma_n\over\alpha_n}\Bigr)^2
           +{1\over\alpha_n}
           +{\hat\m{2+r}_n\over\hat\sigma^{2+r}_n}
                \Bigl({\hat\sigma_n\over\alpha_n}\Bigr)^{2+r}
           +{\hat\m{2+r}_n\over\hat\sigma^{2+r}_n}
                \Bigl({\hat\sigma_n\over\alpha_n}\Bigr)^{\xi_n-1};
\] 
\item for any finite $\ell$, such that $\*F\0_n\cap[1,\ell]=\emptyset$,
\begin{eqnarray}\label{eq:gamma:marg}
\lefteqn{\PP(Q_{1,n}=q_1,\ldots,Q_{\ell,n}=q_\ell)}\qqqq\\
  &=&\prod_{k=1}^\ell\PP(X_{k,n}=q_k)
    \biggl[1+O(\eps_n)
           +O\biggl({\sum_{k=1}^\ell m_{k,n}-q_k\over\alpha_n}\biggr)\biggr].
                                               \nonumber
\end{eqnarray}
\item for any $j\not\in\*F\0_n$, 
\begin{equation}\label{eq:gamma:mean}
\EE Q_{j,n} 
  = \EE X_{j,n} \biggl[1+O(\eps_n)
        +O\biggl({\sigma_{j,n}^2+m_{j,n}^2\over m_{j,n}\alpha_n}\biggr)\biggr],
\end{equation}
\item for any $j\in\*F\0_n$, 
\begin{equation}\label{eq:gamma:mean:sat}
\EE Q_{j,n} 
  = \EE X_{j,n} \Bigl[1+O(\eps_n)\Bigr],
\end{equation}

\end{enumerate}
\end{thm}

\begin{proof}{}
The proof of \romi-\romiii follows essentially along the same lines as
for Theorem~\ref{thm:stab}, while \romiv depends on
Equation~(\ref{eq:mean}) of Lemma~\ref{lem:represent}.
\end{proof}

\section{Scaling}\label{sec:scal}

As said in the introduction, this section provides guidelines for
using the above technical results in two ways.
\begin{itemize}
\item {\em Quantitative\/} estimates for the error terms (w.r.t.\ some
limiting distribution), explicitly obtained from the original data
(e.g.\ the total number of customers $m_n$).
\item {\em Qualitative\/} understanding of the ``critical'' values for
$m_n$ which, in some sense, induce phase transitions of interest.
\end{itemize}

The queues are partitioned as follows:
\begin{eqnarray*}
\*F_n &\egaldef& 
  \Bigl\{0\leq k\leq n\ :\ \liminf_{q\to\infty}
   \sqrt[q]{\mu_{k,n}(1)\cdots\mu_{k,n}(q)}<\infty\Bigr\},\\
\*I_n &\egaldef& 
  \Bigl\{0\leq k\leq n\ :\ \liminf_{q\to\infty}
   \sqrt[q]{\mu_{k,n}(1)\cdots\mu_{k,n}(q)}=\infty\Bigr\}.
\end{eqnarray*}

 From the general discussion at the beginning of
Section~\ref{sec:model}, $\*F_n$ is never empty. Let also
\[
\mu_{k,n} 
  \;\egaldef\; \cases{
    \displaystyle\liminf_{q\to\infty}\sqrt[q]{\mu_{k,n}(1)\cdots\mu_{k,n}(q)}, 
                  & if $k\in\*F_n$,\cr
    \mu_{k,n}(1), & if $k\in\*I_n$,}
\]

\[ 
\rho_{k,n} 
  \egaldef {\lambda_n\pi_{k,n}\over\mu_{k,n}},\qquad
\lambda\0_n
  \egaldef \min_{k\in\*F_n}{\mu_{k,n}\over\pi_{k,n}},\qquad
\rho\0_n 
  \egaldef \max_{k\in\*F_n} \rho_{k,n}={\lambda_n\over\lambda\0_n}.
\]

We shall also need the following subset of $\*F_n$:
\[
 \*F\0_n \egaldef \{ k\in\*F_n\ :\ \rho_{k,n}=\rho\0_n\}.
\]

Note that the definitions of $\mu_{k,n}$ and $\rho\0_n$ are consistent
with the discussion which lead to~(\ref{eq:defrhon}). Moreover, in
most practical cases, $\mu_{k,n}(q)\to\mu_{k,n}$ as $q\to\infty$,
provided that this limit exists and is finite. 

To avoid uninteresting technicalities, it will be convenient to
introduce Assumptions~\ref{ass:uan} and~\ref{ass:service}, but it
should be pointed out that the results of Section~\ref{sec:llt} are
valid in a more general setting. Simple conditions ensuring
\ref{ass:pole} and \ref{ass:service} are discussed in
Section~\ref{sec:ass}.

\begin{ass}\label{ass:uan}
The following limit holds:
\[
 \lim_{n\to\infty} 
  \max_{1\leq k\leq n}{{\pi_{k,n}\over\mu_{k,n}}
   \over {\pi_{1,n}\over\mu_{1,n}}+\cdots+{\pi_{n,n}\over\mu_{n,n}}}=0.
\]
\end{ass}

Assumption~\ref{ass:uan} is somehow unavoidable to obtain a meaningful
asymptotic behaviour of the network. It says that it is possible to
let $m_n\to\infty$ as $n\to\infty$, without saturating the network
and, under the forthcoming Assumption~\ref{ass:service}, it amounts to
Lyapounov's condition~(\ref{eq:lyap}). Note that, when
$\mu_{k,n}=\Omega(1)$ uniformly in $k$ and $n$, \ref{ass:uan} is
simply equivalent to
\[
 \lim_{n\to\infty}\max_{1\leq k\leq n}\pi_{k,n}=0.
\]

\begin{ass}\label{ass:service}
\begin{enumerate}
\item For any real $A<1$ and any integer $r\leq4$, and for any $k\in\*F_n$
such that $\rho_{k,n}\leq A$,
\begin{equation}\label{eq:service:omega}
  m_{k,n}=\Omega(\rho_{k,n}),\ \ \m{r}_{k,n}=\Omega(\rho_{k,n}),
  \ \ \gamma^2_{k,n}=\Omega(\rho_{k,n})
\end{equation}
uniformly in $k$ and $n$.
\item (\ref{eq:service:omega}) also holds for all $k\in\*I_n$.
\end{enumerate}
\end{ass}

The derivation of the most general results of the section is done in
Lemma \ref{lem:m0n} and Theorem~\ref{thm:m0n}. Further insight, under
some additional assumptions, is presented in Theorems~\ref{thm:scal}
and~\ref{thm:xiinf}.

\begin{defi}
A sequence $m\0_n$ is said to be {\em weakly critical\/} for
$\*C_n$ if, for any $0<t<1$,
\begin{equation}\label{eq:defg}
 g(t)\egaldef\limsup_{n\to\infty} {m_n(t\lambda\0_n)\over m\0_n}
\end{equation}
exists and $\displaystyle\lim_{t\to1-}g(t)$ be either $1$ or
$\infty$.

If, in addition, the relation
\[
 \lim_{t\to1-}\liminf_{n\to\infty} {m_n(t\lambda\0_n)\over m\0_n}
 =\lim_{t\to1-}\limsup_{n\to\infty} {m_n(t\lambda\0_n)\over m\0_n},
\]
holds, then the sequence is said to be {\em strongly critical\/} for
$\*C_n$.
\end{defi}

Before seeing how such critical sequences can be used, the next lemma
proves their existence. 

\begin{lem}\label{lem:m0n}
Under assumption~\ref{ass:service}, a convenient weakly critical
sequence for $\*C_n$ is, for some fixed $0<u<1$,
\begin{equation}\label{eq:m0nu}
  m\0_n(u)\egaldef h_um_n(u\lambda\0_n), 
\end{equation}
where $h_u$ is correctly chosen.
\end{lem}

\begin{proof}{}
Choose $(t,u)\in]0,1[\times]0,1[$. From~\ref{ass:service},
\[
 m_n(t\lambda\0_n)
   =\Omega(m_n(u\lambda\0_n))
   =\Omega\biggl(\sum_{k=1}^n{t\lambda\0_n\pi_{k,n}\over\mu_{k,n}}\biggr),
\]
and the application $t\mapsto m_n(t\lambda\0_n)/m_n(u\lambda\0_n)$
is increasing and locally bounded. Therefore,
\[
 \hat g_u(t)\egaldef \limsup_{n\to\infty} {m_n(t\lambda\0_n)
                                 \over m_n(u\lambda\0_n)}
\]
exists and is increasing. To conclude the proof, take
\[
 h_u=\cases{\displaystyle\lim_{t\to1-}\hat g_u(t), & if the limit is finite,\cr
             1, & otherwise.}
\]

It is interesting to note that, if the above limit is finite for some
$u$, it is finite for all $u\in]0,1[$. The proof of the lemma is
concluded.
\end{proof}

In fact, as shown in Theorem~\ref{thm:m0n}, any critical sequence
$m\0_n$ acts as a threshold parameter for $m_n$.  Under~\ref{ass:uan}
and~\ref{ass:service}, which are satisfied by a wide variety of
networks, we provide a nearly complete classification in terms of {\em
necessary and sufficient\/} scaling.  It is worth to emphasize that
any $m\0_n$ chosen from~(\ref{eq:m0nu}) has a {\em pseudo-explicit\/}
form, given in terms of the data of the original network.

\medskip
The second step is to enumerate in a consistent way the desirable
properties of the distribution of $Q_{1,n},\ldots,Q_{n,n}$: for some
finite $j$ and some unspecified $\eps_n$, such that $\eps_n\to0$ as
$n\to\infty$, we have
\begin{eqnarray}
\EE Q_{j,n}
  &=& \EE X_{j,n}\Bigl[1+O(\eps_n)\Bigr], \label{eq:scal:mean}\\[0.5ex]
\PP(Q_{1,n}=q_1,\ldots,Q_{j,n}=q_j)
  &=&\prod_{k=1}^j\PP(X_{k,n}=q_k)
 \Bigl[1+O(\eps_n)\Bigr], \label{eq:scal:marg}
\end{eqnarray}
and also, when Theorem~\ref{thm:stab} [resp. Theorem~\ref{thm:pole}] holds,
the following equation (\ref{eq:scal:normal}) [resp. (\ref{eq:scal:gamma})]:

\begin{eqnarray}
\lefteqn{\PP(Q_{1,n}=q_1,\ldots,Q_{n,n}=q_n)}\qqqq\nonumber\\
  &=&\sqrt{2\pi}\sigma_n\prod_{k=1}^n\PP(X_{k,n}=q_k)
        \Bigl[1+O(\eps_n)\Bigr], \label{eq:scal:normal}\\[0.5ex]
\lefteqn{\PP(Q_{1,n}=q_1,\ldots,Q_{n,n}=q_n)}\qqqq\nonumber\\
  &=&{\alpha_n\Gamma(\xi_n)\over\xi_n^{\xi_n-1}e^{-\xi_n}}
      \prod_{k=1}^n\PP(X_{k,n}=q_k)\Bigl[1+O(\eps_n)\Bigr].
                                   \label{eq:scal:gamma}
\end{eqnarray}

\begin{thm}\label{thm:m0n}
Let~\ref{ass:uan} and~\ref{ass:service} hold and $m\0_n$ be a weakly
critical sequence for $\*C_n$, with the associated function $g(t)$.

Assume first that $\lim_{t\to1-}g(t)=1$. Then the following
classification holds:
\begin{enumerate}
\item If 
\[  
  \limsup_{n\to\infty} {m_n\over m\0_n}<1,
\]
then (\ref{eq:scal:mean}), (\ref{eq:scal:marg}) and
(\ref{eq:scal:normal}) hold with $\eps_n=1/m_n$. In particular $\EE
Q_{k,n}$ is bounded, uniformly in $k$ and $n$.
\item If 
\[  
  \limsup_{n\to\infty} {m_n\over m\0_n}>1,
\]
then, for any sequence of queues $k_n$ in $\*F\0_n$, we have
$\displaystyle\limsup_{n\to\infty}\EE Q_{k_n,n}=\infty$.
\item If $m\0_n$ is a strongly critical sequence and
\[  
  \liminf_{n\to\infty} {m_n\over m\0_n}>1,
\]
then, for any sequence of queues $k_n$ in $\*F\0_n$, we have
$\displaystyle\lim_{n\to\infty}\EE Q_{k_n,n}=\infty$.
\end{enumerate}

In the situation $\lim_{t\to1-}g(t)=\infty$, the same results hold,
just replacing {\rm ``$<1$''} (resp.\ {\rm ``$>1$''}) in the r.h.s.\
of the inequalities by {\rm ``$<\infty$''} (resp.\ {\rm
``$=\infty$''}).
\end{thm}

\begin{proof}{}
To prove~\romi, note that
$m_n=m_n(\lambda_n)=m_n(\rho\0_n\lambda\0_n)$. Since
$m_n(t\lambda\0_n)$ is increasing in $t$, this implies that, when
$\limsup_{n\to\infty}\rho\0_n=1$, we have also
$\limsup_{n\to\infty}m_n/m\0_n\geq1$.  Therefore, in case~\romi, there
exists $\tau<1$ such that $\rho\0_n\leq\tau$ for any
$n\in\NN$. Using~\ref{ass:service}, we can estimate all error terms
coming in Theorem~\ref{thm:stab} and the result is proved. 

Similarly in case~\romii [resp.\ \romiii], we have necessarily
$\limsup_{n\to\infty}\rho\0_n=1$ [resp.\
$\lim_{n\to\infty}\rho\0_n=1$], and the result follows from the
monotonicity of the function $t\mapsto m_{k_n,n}(t\lambda\0_n)$.

The case $\lim_{t\to1-}g(t)=\infty$ is handled with the same method.
\end{proof}

Direct applications of Theorem~\ref{thm:m0n} are proposed farther on
in sections~\ref{ssec:MalYak} and~\ref{ssec:tight}.

In order to get finer results, the next assumption ensures that the
queues not belonging to $\*F\0_n$ stay uniformly away from saturation
conditions.

\begin{ass}\label{ass:nonsat}
There exists a constant $A<1$ such that, 
\begin{equation}
 \lambda\0_n{\pi_{k,n}\over\mu_{k,n}}
    \leq A, \mbox{ \ \ for all }k\in\*F_n\setminus\*F\0_n,\label{eq:scal:A}\\
\end{equation}
\end{ass}

In order to properly reformulate the results of Section~\ref{sec:llt},
let us define
\begin{eqnarray}
 \hat m_n(\lambda)
  &\egaldef& \sum_{k\not\in\*F\0_n}m_{k,n}(\lambda),\\
\hat m\0_n
  &\egaldef& \hat m_n(\lambda\0_n).\label{eq:defm0n}
\end{eqnarray}

Using~(\ref{eq:derm}), it is not difficult to see that $\hat m\0_n$
defined~(\ref{eq:defm0n}) is a strongly critical sequence for $\*C_n$
under \ref{ass:pole}, \ref{ass:uan}, \ref{ass:service} and
\ref{ass:nonsat}. Therefore, all results of Theorem~\ref{thm:m0n}
hold, as well as the following:

\begin{thm}\label{thm:scal}
Let \ref{ass:pole}, \ref{ass:uan}, \ref{ass:service} and
\ref{ass:nonsat} hold. If $\xi_n$ is uniformly bounded, then the
following results hold:
\begin{enumerate}
\item If there exists $\theta_n>0$, such that, for all $n\in\NN$,
\[
{m_n\over \hat m\0_n}\leq1-\theta_n,
\]
and $\lim_{n\to\infty}\theta_n^2m_n=\infty$, then
(\ref{eq:scal:mean}), (\ref{eq:scal:marg}) and (\ref{eq:scal:normal})
hold with
\[
 \eps_n\egaldef{1\over m_n}+{1\over m_n^2\theta_n^4},
\]
except when a queue in $\*F\0_n$ is concerned, in which case
(\ref{eq:scal:mean}) and (\ref{eq:scal:marg}) hold with
\[
 \eps_n={1\over\theta_n^2m_n}.
\]

\item If there exists $\theta_n>0$, such that, for all $n\in\NN$,
\[
{m_n\over \hat m\0_n}\geq1+\theta_n,
\]
and
$\lim_{n\to\infty}\theta_n\hat m\0_n=\lim_{n\to\infty}\theta_n^2\hat m\0_n=\infty$,
then (\ref{eq:scal:mean}) and (\ref{eq:scal:gamma}) hold with
\[
 \eps_n\egaldef{1\over \theta_n^2\hat m\0_n}+{1\over \theta_n\hat m\0_n}
        +{1\over\sqrt{\hat m\0_n}}\biggl[{1\over \hat m\0_n\theta_n^2}\biggr]^{\xi_n-1\over2}.
\]

Moreover, if in Equation~(\ref{eq:scal:marg}),
$[1,j]\cap\*F\0_n=\emptyset$, then the latter also holds,
with $\eps_n$ having the above value.
\end{enumerate}
\end{thm}

\begin{proof}{}
To prove \romi, note that when $m_n\leq (1-\theta_n)\hat m\0_n$, 
\[
  \hat m_n(\lambda_n)\leq m_n(\lambda_n)\leq (1-\theta_n)\hat m_n(\lambda\0_n).
\]

Moreover, using~\ref{ass:service}, \ref{ass:nonsat} and
(\ref{eq:derm}), Taylor's formula yields, for some
$\lambda\in]\lambda_n,\lambda\0_n[$,
\begin{eqnarray*}
 \hat m\0_n-\hat m_n(\lambda_n)
 &=& \hat m_n(\lambda\0_n)-\hat m_n(\lambda_n)\\
 &=& (\lambda\0_n-\lambda_n){\hat\sigma_n^2(\lambda)\over\lambda}\\
 &=& (\lambda\0_n-\lambda_n)\Omega\biggl(\sum_{k\in\*F_n\setminus\*F\0_n}
                                            {\pi_{k,n}\over\mu_{k,n}}\biggr),
\end{eqnarray*}
which implies 
\[
 1-{\hat m_n(\lambda_n)\over \hat m\0_n}=\Omega(1-\rho\0_n)\geq\theta_n.
\]

Hence, 
\[
 {1\over1-\rho\0_n}=O\biggl({1\over\theta_n}\biggr)
\]
and, using $\hat m_n(\lambda_n)=\Omega(\hat
\m{r}_n)=\Omega(\hat\sigma_n^2)$, a direct but tedious computation
shows that Theorem~\ref{thm:stab} applies with appropriate error
terms.

Let us now prove assertion \romii. It follows from
\[
 m_n-\hat m_n(\lambda_n)=\Omega\Bigl({\xi_n\rho\0_n\over1-\rho\0_n}\Bigr)
 \geq \theta_n\hat m\0_n\to\infty,
\]
that $\rho\0_n\to1$ and
\begin{eqnarray*}
 {\hat\sigma_n^2\over\alpha_n^2}
 &=& \Omega(\hat m_n(\lambda_n)(1-\rho\0_n)^2)\\
 &=& O\Bigl({\hat m_n(\lambda_n)\over (m_n-\hat m_n(\lambda_n))^2}\Bigr)
 \;=\; O\Bigl({\hat m_n(\lambda_n)\over \theta_n^2[\hat m\0_n]^2}\Bigr)
 \;=\; O\Bigl({1\over \theta_n^2\hat m\0_n}\Bigr).
\end{eqnarray*}

Thus, Theorem~\ref{thm:pole} applies and \romii is proved.
\end{proof}

It remains to state what happens when $\xi_n\to\infty$ as
$n\to\infty$. As shown below, this behaviour does not depend on the
saturation of the queues in $\*F\0_n$.

\begin{thm}\label{thm:xiinf}
Let $\xi_n\to\infty$ as $n\to\infty$. Let also \ref{ass:pole},
\ref{ass:uan}, \ref{ass:service} and \ref{ass:nonsat} hold.  Then,
under the uniformity assumption
\begin{equation}\label{eq:xiinf:m4}
 \m4_{k,n}=O\left({\xi_{k,n}\rho\0_n\over(1-\rho\0_n)^4}\right), 
  \mbox{ for all }k\in\*F\0_n,
\end{equation}
the results (\ref{eq:scal:mean}), (\ref{eq:scal:marg}) and
(\ref{eq:scal:normal}) are again valid, with
\[
 \eps_n\egaldef{1\over (1-\rho\0_n)m_n}.
\]
\end{thm}

\begin{proof}{}
The statement relies on Theorem~\ref{thm:stab}, taking $r=2$. First,
from classical weak compactness and moment convergence theorems (see
e.g.~\cite{Loe:1}), it follows that, for $k\in\*F\0_n$ and all
$0<s\leq4$
\[
 \m{s}_{k,n}=\Omega\left({\xi_{k,n}\rho\0_n\over(1-\rho\0_n)^s}\right).
\]

Thus, the term coming in Lyapounov's condition~(\ref{eq:lyap}) is
equal to
\begin{eqnarray*}
{\m4_n\over\sigma^4_n}
  &=& \Omega\left({\hat m_n+{\rho\0_n\xi_n\over(1-\rho\0_n)^4}
               \over \left[\hat m_n
                           +{\rho\0_n\xi_n\over(1-\rho\0_n)^2}
                     \right]^2}\right)\\
  &=& \Omega\left({(1-\rho\0_n)^4\hat m_n+\rho\0_n\xi_n
                  \over\left[(1-\rho\0_n)^2\hat m_n+\rho\0_n\xi_n
                      \right]^2}\right)\\
  &=& O\left({1\over (1-\rho\0_n)\hat m_n+\rho\0_n\xi_n}\right)\\
  &=& O\left({1\over (1-\rho\0_n)m_n}\right),
\end{eqnarray*}
which tends to $0$ as $n\to\infty$. The other error terms
given in Theorem~\ref{thm:stab} are estimated in the same way.

The only thing left to check is that $\sigma_n^2=O(\gamma_n^2)$. In
fact, since $\gamma_n^2=\Omega(\hat m_n)$, this relation will only
hold when $\rho\0_n$ is uniformly bounded away from~$1$. However, for
any $k\in\*F\0_n$ and for any $\theta\in[-\pi,\pi]$,
\begin{eqnarray*}
 |\phi_{k,n}(\theta)|
  &=& |\omega_{k,n}(\theta)|^{\xi_{k,n}}\Bigl|1+O(\theta)\Bigr|\\
  &\leq& \left[{1\over1+{\alpha_n^2\theta^2\over6}}\right]^{\xi_{k,n}\over2}
           \Bigl|1+O(\theta)\Bigr|\\
  &\leq& \left[{1\over1+{\alpha_n^2\theta^2\over6}}\right]^{\xi_{k,n}\over4},
\end{eqnarray*}
provided that $a<\rho\0_n<1$, where $a$ is some fixed constant. This
bound can be used to replace Equation~(\ref{eq:majphin}) in the proof
of Proposition~\ref{pro:clt:norm} by
\begin{eqnarray*}
 \left|\int_{\delta_n\leq |\theta|\leq\pi}e^{-\ci\theta x}
       \phi_n(\theta)d\theta\right|
  &\leq& \int_{|\theta|\geq\delta_n}
    \left[{1\over1+{\alpha_n^2\theta^2\over6}}\right]^{\xi_n\over4}d\theta\\
  &=& O\biggl({1\over\delta_n\alpha_n^2\xi_n}
            {1\over(1+\alpha_n^2\delta_n^2)^{{\xi_n\over4}-1}}\biggr),
\end{eqnarray*}
which is exponentially small in $\xi_n$, since $\delta_n\alpha_n=\Omega(1)$.
\end{proof}

\section{Towards more tangible assumptions}\label{sec:ass}

The assumptions used in the results of the previous section may seem
difficult to check in practice. However, as shown hereafter,
they can be replaced (at the expense of a loss in generality) by
simpler properties directly related to the service mechanisms of the
queues.

The next lemma provides a realistic context in which~\ref{ass:service}
is satisfied.

\begin{lem}\label{lem:ass:service}
Assume that 
\begin{enumerate}
\item there exist sequences $R(q)$ and $T(q)$ such that 
\[
  \liminf_{q\to\infty} \sqrt[q]{R(1)\cdots R(q)}=1,
\]
\[
  \liminf_{q\to\infty}T(q)=\infty,
\]
and, for any $q>0$,
\begin{eqnarray*}
  \mu_{k,n}(q) &\geq& R(q)\mu_{k,n}, \mbox{\quad for }k\in\*F_n,\\
  \mu_{k,n}(q) &\geq& T(q)\mu_{k,n}, \mbox{\quad for }k\in\*I_n;
\end{eqnarray*}
\item there exists a constant $B<\infty$ such that
\[
 \lambda\0_n{\pi_{k,n}\over\mu_{k,n}}<B,\mbox{ for all }k\in\*I_n.
\]
\end{enumerate}
Then~\ref{ass:service} holds.
\end{lem}

\noindent{\bf Remark\ } This lemma can  be applied in particular to any
mixing of $M/M/\infty$ and multiple-server queues with at most $c$
servers, with
\[
 R(q)=\min\left[1,{q\over c}\right],\ \ T(q)=q.
\]

\begin{proof}{}
For each queue $k\in\*F_n$ such that $\rho_{k,n}\leq A$, and for all
$r\in\NN$, we have
\[
 \sum_{q=0}^\infty q^r{(\lambda_n\pi_{k,n})^q
   \over\mu_{k,n}(1)\cdots\mu_{k,n}(q)}
 \leq \sum_{q=0}^\infty {q^rA^q\over R(1)\cdots R(q)}<\infty.
\]

In particular, $f_{k,n}(\lambda_n\pi_{k,n})=\Omega(1)$ and 
\[
m_{k,n}
  = {\lambda_n\pi_{k,n}\over\mu_{k,n}(1)f_{k,n}(\lambda_n\pi_{k,n})}
      \sum_{q=1}^\infty q {(\lambda_n\pi_{k,n})^{q-1}
                       \over\mu_{k,n}(2)\cdots\mu_{k,n}(q)} 
  = \Omega\biggl({\lambda_n\pi_{k,n}\over\mu_{k,n}(1)}\biggr).
\]

Similarly, for any $r\in\NN$, 
\[
 \m{r}_{k,n}=\Omega\biggl({\lambda_n\pi_{k,n}\over\mu_{k,n}(1)}\biggr).
\]

The same computations can be applied to $k\in\*I_n$, thus
proving~\ref{ass:service}-\romii.
\end{proof}

The results of Section~\ref{sec:scal} can be easily generalized to a
situation where some $M/M/\infty$ queues of $\*I_n$ become saturated,
in which case~\ref{ass:service}-\romii is no longer satisfied.
Indeed, the characteristic function of the number of clients $X$ in an
$M/M/\infty$ queue with parameter $\rho$ can be written as
\[
 \EE e^{\ci\theta X} 
   = \exp\Bigl(\rho(e^{\ci\theta}-1)\Bigr)
   = \biggl[\exp\Bigl({\rho\over\lfloor\rho\rfloor}
                      (e^{\ci\theta}-1)\Bigr)\biggr]^{\lfloor\rho\rfloor},
\]
which means that a saturated infinite server queue can be replaced by
several non-saturated infinite-server queues without changing the
distribution of $S_n$. Therefore, the results of
Section~\ref{sec:scal} still hold, except for marginal distributions
containing one of the saturated queues.

\medskip
Theorems~\ref{thm:scal} and~\ref{thm:xiinf} also required
assumption~\ref{ass:pole} on the service mechanisms of the so-called
``saturable'' queues. It is often enough to restrict ourselves to the
following two categories of queues, which encompass the standard
$M/M/c$ queue.

\begin{lem}\label{lem:ass:pole}
Assume that, for any $k\in\*F\0_n$, either
\begin{itemize}
\item[\romi] there is a constant $q_c$, independent of $k$ and
$n$, such that
\begin{equation}\label{eq:lap:mmc}
 {\mu_{k,n}(q)\over\mu_{k,n}}=
   \cases{O(1), & if $q<q_c$,\cr
          1,    & otherwise.}
\end{equation}
\end{itemize}
or
\begin{itemize}
\item[\romii] for some finite constants $\xi_{\min}$ and $\xi_{\max}$, 
\begin{equation}\label{eq:lap:xi}
 \log{\mu_{k,n}(q)\over\mu_{k,n}}=-{\xi_{k,n}-1\over q}+\Delta_{k,n}(q),
\end{equation}
with 
\[
  \Delta_{k,n}(q)=O\left({1\over q^2}\right),\qquad
 1<\xi_{\min}\leq\xi_{k,n}\leq\xi_{\max}\,,
\]
uniformly in $k$ and $n$. (See also Section \ref{sec:concl}).
\end{itemize}

Then \ref{ass:pole} holds. 
\end{lem}

\begin{proof}{}
In view of Equation~(\ref{eq:defphikn}), for any fixed $k$
and $n$, the quantity to estimate is related to
\begin{eqnarray*}
  f_{k,n}(\lambda_n\pi_{k,n}e^{\ci\theta})
  &=& \sum_{q=0}^\infty {(\lambda_n\pi_{k,n}e^{\ci\theta})^q
                        \over\mu_{k,n}(1)\cdots\mu_{k,n}(q)}\\
  &=& \sum_{q=0}^\infty \prod_{p=1}^q{\mu_{k,n}\over\mu_{k,n}(p)}
                (\rho_{k,n}e^{\ci\theta})^q.
\end{eqnarray*}

For the sake of brevity, let us omit the $k$ and $n$ subscripts and
define, for any $z\in\CC$, $|z|<1$,
\[
 g(z)\egaldef \sum_{q=0}^\infty \prod_{p=1}^q{\mu\over\mu(p)}z^q.
\]

Thus, we have to estimate $g(\rho e^{\ci\theta})/g(\rho)$, for
$\theta\in[-\pi,\pi]$ and $\rho<1$. This proof proceeds in steps:

\begin{itemize}
\item[a)] Assume first that~(\ref{eq:lap:mmc}) holds. Then
\[
  g(z)={1\over 1-z}\left[
         O(1-z)\sum_{q=0}^{q_c}\prod_{p=1}^q{\mu\over\mu(p)} z^q
         +z^{q_c+1}\prod_{p=1}^{q_c}{\mu\over\mu(p)}\right],
\]
and Assumption~\ref{ass:pole} holds with $\xi=1$.

\item[b)] 
Under (\ref{eq:lap:xi}), one obtains, for $q\geq1$, 
\begin{eqnarray*}
 \prod_{p=1}^q{\mu\over\mu(p)}
   &=& \exp\left[(\xi-1)\sum_{p=1}^q{1\over p}-\sum_{p=1}^q\Delta(p)\right]\\
   &=& \exp[(\xi-1)C-\Delta]\cdot q^{\xi-1}[1+{\textstyle{a_q\over q}}],
\end{eqnarray*}
where $C$ is the Euler constant,
$\Delta\egaldef\sum_{p=1}^\infty\Delta(p)$, and $a_q$ is uniformly
bounded. In the remainder of the proof, let 
\[
 K\egaldef\exp[(\xi-1)C-\Delta].
\]

\item[c)] Let, for $|z|<1$ and $s\in\CC$,  
\[ \varphi(z,s)\egaldef\sum_{q=1}^\infty {z^q\over q^s}.\]

Then, for $\Re(s)>0$, 
\[ \varphi(z,s)={z\over\Gamma(s)}\int_0^\infty {t^{s-1}dt\over e^t-z}. \]

In fact, this integral representation can be used to get an analytic
continuation with respect to $s$, by introducing the (classical)
Hankel's contour. This yields, for all $|z|<1$ and $\Re(s)>0$,
\[
 \varphi(z,s)
   ={\ci\Gamma(1-s)\over2\pi}\int_{\*L} {(-t)^{s-1}dt\over e^t-z}.
\]

\newcommand\ZZstar{\ZZ\setminus\{0\}}

Distorting $\*L$ to include the zeros of $e^t-z$, the following
expression holds, for $\Re(s)<0$ and all values of $z$ such that
$|\arg(-\log z+2\ci n\pi)|\leq\pi$:
\[
 \varphi(z,s)=\Gamma(1-s)\sum_{n\in\ZZ}(-\log z+2\ci n\pi)^{s-1}.
\]

\item[d)] Using this expression, simple computations yield, when
$\xi>1$ and $|z|<1$
\begin{eqnarray*}
g(z)
  &=& 1+K\left[\varphi(z,1-\xi)+\sum_{q=1}^q q^{\xi-2}a_q z^q\right]\\
 &=& {K\over\log^\xi z}\Biggl[
       {\log^\xi z\over K}+
        1+\sum_{n\neq0}\left({\log z\over
              \log z-2\ci n\pi}\right)^\xi\\
 & &\hspace{3.2em} +\log^\xi z\sum_{q=1}^q q^{\xi-2}a_q z^q\Biggr],
\end{eqnarray*}
and, finally,
\[
 {g(\rho e^{\ci\theta})\over g(\rho)}
   =\left[{1-\rho\over1-\rho e^{\ci\theta}}\right]^\xi
         \left[1+\xi\rho(e^{\ci\theta}-1)\right].
\]
\end{itemize}
This concludes the proof of the lemma.
\end{proof}

\section{Applications}\label{sec:appl}

\subsection{A Jackson network with convergence properties}
\label{ssec:MalYak}

Consider the basic Jackson network (consisting of $M/M/1$ queues with
constant service rates) analyzed in~\cite{MalYak:1}.

In this case, 
\[
  m_n(t\lambda\0_n)
  \;=\; \sum_{k=1}^n{tr_{k,n}\over1-tr_{k,n}},
   \ \ \mbox{ with }r_{k,n}={\lambda\0_n\pi_{k,n}\over\mu_{k,n}}.
\]

Under the assumption made in~\cite{MalYak:1} that the counting measure
\[
 I_n(A)\egaldef{1\over n}\,\mbox{\rm Card}(k : r_{k,n}\in A),
\]
defined for all Borel sets $A$, converges weakly to a probability
measure $I$, we have
\[
 \lim_{n\to\infty}{m_n(t\lambda\0_n)\over n}
   =\int_0^1{tr\over1-tr}dI(r),
\]
and
\[
 \lim_{t\to1-}\int_0^1{tr\over1-tr}dI(r)\egaldef\lambda_{cr}\leq\infty.
\]

Thus, the results of~\cite{MalYak:1} are contained in the theorems of
Section~\ref{sec:scal}, taking $m\0_n=n\lambda_{cr}$, which is then a
{\em strongly critical sequence\/} for $\*C_n$.

\subsection{A network with tight bottlenecks}\label{ssec:tight}

As pointed out in the introduction, there are cases of interest with
$m_n=o(n)$. This will be illustrated in the next example.

Consider a closed network consisting of $s_n$ subnetworks of $M/M/1$
queues having each a unique entry point, in which a fixed number $m$
of tasks circulate. The queues are subject to failures, taking place
with some probability $f<1$. When a failure occurs, the task returns to
the entry point of its current subnetwork. Tasks visit the various
subnetworks according to some probability matrix.

This model exhibits tight bottlenecks, when the number and the size of
the subnetworks grow. This fact, for the sake of simplicity, will be
illustrated on a very simple topology, presented in
Figure~\ref{fig:tandem}: all subnetworks are associated in tandem, and
each of them consists itself of $\ell_n$ queues in tandem, with unit
processing rates.

\begin{figure}
\center
\begin{picture}(125,20)
\def\dashline(#1,#2)#3{
	\multiput(0,0)(#1,#2){#3}{\line(#1,#2){0.5}}}
\def\error{
	\put(0,0){\vector(0,1){3}}
	\put(0,3){\line(0,1){1}}
	\put(-2,2){\makebox(0,0){$\scriptstyle f$}}}
\def\aqueue{
	\put(0,0){\circle*{2}}
	\put(0,1){\error}}
\def\queuelink{
	\put(0,0){\vector(1,0){6}}
	\put(7,0){\aqueue}}
\def\queues{
	\put(0,0){\queuelink}
	\put(8,0){\queuelink}
	\put(16,0){\dashline(1,0){6}}
	\put(23,0){\aqueue}
	\put(24,0){\queuelink}}
\def\subnet#1{
	\put(0,0){\dashbox{1}(35,10)[bl]{
		\put(0,3){\queues}
		\put(9,8){\line(-1,0){6}}
		\put(15,8){\vector(-1,0){6}}
		\put(15,8){\dashline(1,0){8}}
		\put(31,8){\vector(-1,0){5}}
		\put(26,8){\line(-1,0){3}}
		\put(3,8){\vector(0,-1){3}}
		\put(3,5){\line(0,-1){2}}
		\put(32,3){\line(1,0){3}}}}
	\put(17.5,13){\makebox(0,0){subnet #1}}}

\put(0,3){\line(1,0){2}}
\put(2,0){\subnet{1}}
\put(37,3){\line(1,0){5}}
\put(42,0){\subnet{2}}
\put(77,3){\dashline(1,0){11}}
\put(88,0){\subnet{$s_n$}}
\put(123,3){\line(1,0){2}}
\put(125,3){\vector(0,1){9}}	\put(125,12){\line(0,1){7}}
\put(125,19){\vector(-1,0){19}}	\put(107,19){\line(-1,0){19}}
\put(77,19){\dashline(1,0){11}}
\put(77,19){\vector(-1,0){38.5}}\put(38.5,19){\line(-1,0){38.5}}
\put(0,19){\vector(0,-1){9}}	\put(0,10){\line(0,-1){7}}
\end{picture}
\caption{a compound network of tandem queues}\label{fig:tandem}
\end{figure}
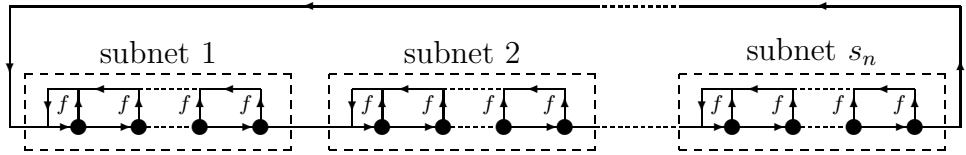

Here, the invariant measure of the routing matrix has the form
\[
  (\pi_{1,n},\ldots,\pi_{{\ell_n},n};\pi_{1,n},\ldots;\ldots,\pi_{{\ell_n},n}),
\]
where $\pi_{k,n}$ is the invariant probability associated to the $k$-th
queue of an arbitrary subnetwork.  A straightforward
computation, using symmetry properties, yields, for any $t\in]0,1[$,
\begin{eqnarray*}
 \pi_{k,n}&=&{1\over s_n}{f(1-f)^{k-1}\over 1-(1-f)^{\ell_n}}
          \;=\; (1-f)^{k-1}\pi_{1,n},\\
 m_n(t\lambda\0_n)
   &=& s_n\sum_{k=1}^{\ell_n}{t(1-f)^{k-1}
                             \over1-t(1-f)^{k-1}}.
\end{eqnarray*}

Choosing some fixed $u\in]0,1[$ and assuming that
$\ell_n\to\infty$ as $n\to\infty$, we have
\[
\lim_{n\to\infty} {m_n(t\lambda\0_n)\over m_n(u\lambda\0_n)}
  = {L_f(t)\over L_f(u)},
\]
where $L_f$ is defined on $]0,1[$ as
\[
 L_f(t)
   \egaldef \sum_{k=1}^\infty{t(1-f)^{k-1}\over1-t(1-f)^{k-1}}
\]
and $\lim_{t\to1-}L_f(t)=\infty$. 

Therefore, $m_n(u\lambda\0_n)$ is a strongly critical sequence for the
network and the size of the queues remain uniformly bounded if, and
only if,
\[
 m_n=O\left(m_n(u\lambda\0_n)\right)=O(s_n)=o(n).
\]

\subsection{A service vehicle network}\label{ssec:vehic}

Consider a fleet of vehicles serving an area consisting of $n$
stations forming a fully connected graph.  These vehicles are used to
transport goods or passengers. Vehicles wait at stations until they
receive a request, in which case they go to an other station. The
routing among stations is done according to some routing matrix
$P_n$. When a request arrives to an empty station, it is immediately
lost. The request arrivals form a Poisson stream at each queue.

We model this system as follows: for all $0\leq k\leq n$, station $k$
is represented as a single-server queue with service rate $\mu_{k,n}$
which is equal to the arrival rate at station $k$, since arrivals are
lost when the station is empty. When a vehicle leaves station $k$, it
chooses its destination according to the Markovian routing matrix
$P_n=(p_{k\ell,n})$.  The duration of the journey between two stations
$k$ and $\ell$ is represented by an infinite server queue placed on
the edge between them.  The service rate of this queue when there are
$q$ vehicles traveling between $k$ and $\ell$ is
$q\mu_{k\ell,n}$. Note that, contrary to the convention used
throughout this \chapart, the total number of {\em queues\/} is
$n^2+n$. Let $(\pi_{1,n},\ldots,\pi_{n,n})$ be the invariant measure
of $P_n$, defined as in~(\ref{eq:invmeas}). Then, with obvious
notation, for all $k,\ell\in[1,n]$, for all $\theta\in[-\pi,\pi]$,

\[ \newcommand\ds{\displaystyle}
\begin{array}{rclcrcl}
\rho_{k,n}     
  &\egaldef& \ds{\lambda_n\pi_{k,n}\over2\mu_{k,n}},&&
\rho_{k\ell,n}
  &\egaldef& \ds{\lambda_n\pi_{k,n} p_{k\ell,n}\over2\mu_{k\ell,n}},\\[0.5ex]
m_{k,n}        
  &\egaldef& \ds{\rho_{k,n}\over1-\rho_{k,n}},&&
m_{k\ell,n}     
  &\egaldef& \ds\rho_{k\ell,n},\\[0.5ex]
\phi_{k,n}(\theta) 
  &\egaldef& \ds{(1-\rho_{k,n})e^{-\ci m_{k,n}\theta}
              \over1-\rho_{k,n} e^{\ci\theta}},&&
\phi_{k\ell,n}(\theta) 
  &\egaldef& \ds e^{\rho_{k\ell,n}(e^{\ci\theta}-1-\ci\theta)}.
\end{array}
\]

Define $\*F\0_n$ as in Section~\ref{sec:scal} and assume that its
cardinal is some fixed integer $K\geq1$. Lemmas~\ref{lem:ass:service}
and~\ref{lem:ass:pole} apply, taking $R(q)=1$, $T(q)=q$ and
$\xi_{k,n}=1$ for $q\geq1$ and $k\in\*F\0_n$. Thus,
when~\ref{ass:nonsat} holds, Theorem~\ref{thm:scal} can be used and
estimates of many performance measures can be derived, with
corresponding error terms.

Some questions of interest arise:
\begin{itemize}
\item which maximal efficiency can be expected from this system? 
\item how many vehicles should be provided?
\end{itemize}

To answer these questions, it is convenient to define the {\em loss
probability\/} as
\[
 \*P_{\rm loss}(n) \egaldef {\sum_{k=1}^n \mu_{k,n}\PP(Q_{k,n}=0)
                                \over\sum_{k=1}^n \mu_{k,n}}.
\]

$\*P_{\rm loss}(n)$ is the proportion of customers that are lost
because they arrive at an empty station. This is a good indicator of
the quality of service provided by the network. Under appropriate
conditions as $n\to\infty$:
\begin{eqnarray}
 \*P_{\rm loss}(n) 
 &\sim& {\sum_{k=1}^n \mu_{k,n}\PP(X_{k,n}=0)
                                \over\sum_{k=1}^n \mu_{k,n}} \nonumber\\
 &\sim& 1-{\lambda_n\over2\sum_{k=1}^n\mu_{k,n}}.\label{eq:praxiloss}
\end{eqnarray}

The last expression is a decreasing function of $\lambda_n$, which is
itself bounded by $\lambda\0_n$. Therefore, the minimum loss
probability is attained when $\lambda_n\to\lambda\0_n$; this happens
with
\[
 m_n=(1-\theta_n)\hat m\0_n,\ \ \lim_{n\to\infty}\theta_n=0,
\]
where $\theta_n$ is chosen to satisfy the assumptions of
Theorem~\ref{thm:scal}-\romi. With this choice of $m_n$,
(\ref{eq:praxiloss}) holds with
\[
 \lambda_n=\lambda\0_n(1+O(\theta_n)),
\]
which is asymptotically optimal. Consequently, a ``good'' value for
$m_n$ is $m_n=\hat m\0_n$, and having a number of vehicle proportional to
the number of stations can be a poor choice, especially when some
stations are more loaded than others. These stations act as {\em
bottlenecks\/} of the system, which should be removed by altering the
routing probabilities.

\section{General remarks}\label{sec:concl} 

$\quad$ First, a chief difficulty of the analysis is due to the need
of dealing with rate of convergence and limits of {\em densities}:
this is the field of Berry-Esseen theorems and large deviations.

Secondly, the results have been obtained under several technical
assumptions (especially  {\em uniformity}), which in some
sense are unavoidable. This means precisely that the choice of
conditions slightly different from \ref{ass:pole}, \ref{ass:service}
and \ref{ass:nonsat} would have led to different families of limit
laws having infinitely divisible distributions.

 In particular, from a physical point of view, it is worth commenting
on equation (\ref{eq:lap:xi}). The inequality $\xi_{k,n} \geq 1$
implies that the maximum service rate of the queues in $\*F\0_n$ is
reached from below; this is not the case if $0<\xi_{k,n}<1$, and the
analysis was omitted, since the technicalities involved would have
made the text unnecessarily obscure. At last, the case $\xi_{k,n} \leq
0$ dealing with other types of singularities (for instance
logarithmic), was not carried out, and would yield other limit laws.

The future class of problems of interest concerns some non-product
form networks.

\begin{secappendix}{Appendix}
\subsection{A bound on periodic characteristic functions}
\label{app:bound}

One of the problems arising in the computation of convergence rates in
the Central Limit Theorem is to find upper bounds on the modulus of a
characteristic function $\phi(\theta)$ for $\theta$ away from
$0$. One typical property used can be stated as follows:
\begin{center}\it
 there exist $\theta_0>0$ and $a<1$ such that, for all
$|\theta|>\theta_0$, $|\phi(\theta)|<a$.
\end{center}

It is pointed out in Feller~\cite{Fel:1} that this condition is
usually easy to fulfill in practice, as long as $X$ does not have a
lattice distribution. Unfortunately, we are in the lattice case and thus
must cope with the periodicity of $\phi$.

Next lemma shows how a bound on $|\phi(\theta)|$ can be derived for
$|\theta|\leq\pi$.

\begin{lem}\label{lem:majcar}
Let $X$ be an integer-valued random variable with distribution
$P(X=k)=p_k$, $k\in\NN$. Define
\[
\gamma^2
  \egaldef \sum_{k=0}^\infty {p_{2k}p_{2k+1}\over p_{2k}+p_{2k+1}}
  \leq \min\Bigl(\var X,{1\over4}\Bigr),    
\]
where the summands are taken to be zero when $p_{2k}=p_{2k+1}=0$.
Then, for any $\theta\in[-\pi,\pi]$, the characteristic function $\phi$
of $X$ satisfies:
\begin{equation}\label{eq:majcar}
|\phi(\theta)|\leq \exp\Bigl(-{\gamma^2\over5}\theta^2\Bigr).
\end{equation}
\end{lem}

\begin{proof}{}
We have
\[
|\phi(\theta)| 
  = \Bigl| \sum_{k=0}^\infty p_ke^{\ci k\theta}\Bigr|
  \leq \sum_{k=0}^\infty \Bigl| p_{2k}+p_{2k+1}e^{\ci\theta}\Bigr|.
\]

Moreover,
\begin{eqnarray*}
\Bigl| p_{2k}+p_{2k+1}e^{\ci\theta}\Bigr|
  &=& \sqrt{(p_{2k}+p_{2k+1}\cos\theta)^2+p_{2k+1}^2\sin^2\theta}\\
  &=& \sqrt{(p_{2k}+p_{2k+1})^2-2p_{2k}p_{2k+1}(1-\cos\theta)}\\
  &\leq& p_{2k}+p_{2k+1}
   -{p_{2k}p_{2k+1}\over p_{2k}+p_{2k+1}}(1-\cos\theta).
\end{eqnarray*}

Hence, for $\theta\in[0,\pi]$,
\begin{eqnarray*}
|\phi(\theta)|
   &\leq& 1 - (1-\cos\theta)
     \sum_{k=0}^\infty{p_{2k}p_{2k+1}\over p_{2k}+p_{2k+1}}\\
   &\leq& 1-{2\over\pi^2}\theta^2\gamma^2\\
   &\leq& \exp\bigl(-{2\gamma^2\over\pi^2}\theta^2\bigr),
\end{eqnarray*}
which yields~(\ref{eq:majcar}). That $\gamma^2\leq\var X$ can be seen
by a Taylor expansion of $\phi$ in the neighborhood of $\theta=0$,
while the relation $\gamma^2\leq1/4$ follows from the trivial
inequality
\[
  {p_{2k}p_{2k+1}\over p_{2k}+p_{2k+1}}\leq{p_{2k}+p_{2k+1}\over4}.
\]
\end{proof}

$\gamma$ has the desirable property to be zero when $X$ is an integer
variable with a span strictly greater than $1$, in which case the
period of $\phi$ is less than $2\pi$. Another desirable property would
be that $\gamma\to\infty$ when the moments of $X$ are unbounded; since
$\gamma\leq1/2$, this is obviously not possible here. That this
``feature'' is  somehow unavoidable can be seen on the following example:
\begin{eqnarray*}
\phi(\theta)
  &\egaldef& {2+e^{\ci\theta}\over4}
              +{1\over4}\sum_{k=2}^\infty
                   {e^{\ci k\theta}\over k(k-1)}\\
  &=& {1+e^{\ci\theta}\over2}+(1-e^{\ci\theta})\ln(1-e^{\ci\theta}).
\end{eqnarray*}

The random variable having $\phi$ as characteristic function admits no
finite moment of order greater or equal to $1$, but no bound on
$|\phi|$ is substantially better than~(\ref{eq:majcar}).

\subsection{Proof of Propositions~\protect\ref{pro:clt:norm} 
                          and~\protect\ref{pro:lt:pole}}
\label{app:proof}

\begin{proof}{of Proposition~\ref{pro:clt:norm}}
Using a Fourier inversion formula, the left hand side of
(\ref{eq:clt:norm}) can be rewritten as
\[
 {\sigma_n\over2\pi}\int_{-\pi}^\pi e^{-\ci\theta x}\phi_n(\theta)d\theta
    -{1\over2\pi}\int_{-\infty}^\infty 
         e^{-\ci{x\over\sigma_n}u}e^{-{u^2\over2}}du.
\]

Thus, our goal is to evaluate the quantity
\begin{eqnarray*}
I_n 
  &\egaldef & \int_{-\pi}^\pi e^{-\ci\theta x}\phi_n(\theta)d\theta
    -\int_{-\infty}^\infty 
        e^{-\ci\theta x}e^{-{\sigma_n^2\theta^2\over2}}d\theta\\
  &=& \int_{-\delta_n}^{\delta_n}
        e^{-\ci\theta x}\Bigl(\phi_n(\theta)
                 -e^{-{\sigma_n^2\theta^2\over2}}\Bigr)d\theta\\
  & &{}
      -\int_{|\theta|\geq\delta_n}\!\!e^{-\ci\theta x}
             e^{-{\sigma_n^2\theta^2\over2}}d\theta
      +\int_{|\theta|\in[\delta_n,\pi]}
          e^{-\ci\theta x}\phi_n(\theta)d\theta.
\end{eqnarray*}

It is known that 
\[
 \int_{|\theta|\geq\delta_n}
     e^{-{\sigma_n^2\theta^2\over2}}d\theta
   \approx {2\over\sigma_n^2\delta_n}
              e^{-{\sigma_n^2\delta_n^2\over2}},
\]
applying Lemma~\ref{lem:majcar} to $\phi_n$, we get
\begin{equation}\label{eq:majphin}
\left|\int_{\delta_n\leq |\theta|\leq\pi}
           e^{-\ci\theta x}\phi_n(\theta)d\theta\right|
  \leq \int_{|\theta|\geq\delta_n}
    e^{-{\gamma_n^2\theta^2\over5}}d\theta
  = O\biggl({1\over\gamma_n^2\delta_n}
          e^{-{\gamma_n^2\delta_n^2\over5}}\biggr). 
\end{equation}

Finally, we obtain a bound on $|I_n|$ which is uniform in $x$:
\begin{eqnarray}\label{eq:In}
|I_n| &\leq& \int_{-\delta_n}^{\delta_n}
       \Bigl|\phi_n(\theta)
             -e^{-{\sigma_n^2\theta^2\over2}}\Bigr|d\theta\nonumber\\
    & &{}+ O\biggl({1\over\sigma_n^2\delta_n}
                e^{-{\sigma_n^2\delta_n^2\over2}}\biggr)
      + O\biggl({1\over\gamma_n^2\delta_n}
                e^{-{\gamma^2_n\delta_n^2\over5}}\biggr).
\end{eqnarray}

We proceed now to estimate the above integral, so that implicitly
$|\theta|\leq\delta_n$. The derivation relies on the following simple
inequality, valid for all complex numbers $x_1,\ldots,x_n$ and
$y_1,\ldots,y_n$:
\begin{equation}\label{eq:ineq}
|x_1\cdots x_n - y_1\cdots y_n|
  \leq \sum_{k=1}^n |x_1\cdots x_{k-1}||x_k-y_k||y_{k+1}\cdots y_n|,
\end{equation}
which will be used with $x_k=\phi_{k,n}(\theta)$ and 
$y_k=\exp(-\sigma_{k,n}^2\theta^2/2)$. 

The characteristic function $\phi_{k,n}$ of the random variable
$X_{k,n}$ satisfies (see for example Lo\`eve~\cite{Loe:1})
\begin{equation}\label{eq:majphi}
 \Bigl|\phi_{k,n}(\theta)-1+\sigma_{k,n}^2{\theta^2\over2}\Bigr|
   \leq \m{2+r}_{k,n}{|\theta|^{2+r}\over2}.
\end{equation}

Hence, using the inequality $|e^{-x}-1+x|\leq x^s/s$, valid for
all $x\geq0$ and $1<s\leq2$,
\begin{eqnarray}\label{eq:majdiffk}
\Bigl|\phi_{k,n}(\theta)-e^{-{\sigma_{k,n}^2\theta^2\over2}}\Bigr|
  &\leq& \Bigl|\phi_{k,n}(\theta)-1+\sigma_{k,n}^2{\theta^2\over2}\Bigr|
         +\Bigl|e^{-{\sigma_{k,n}^2\theta^2\over2}}
              -1+\sigma_{k,n}^2{\theta^2\over2}\Bigr|\nonumber\\
  &\leq& \m{2+r}_{k,n}{|\theta|^{2+r}\over2}
         +\sigma_{k,n}^{2+r}{|\theta|^{2+r}\over2}
  \;\leq\; \m{2+r}_{k,n}|\theta|^{2+r}.
\end{eqnarray}

To find an upper bound for $|\phi_{k,n}|$, assume first
$\sigma_{k,n}\delta_n\leq1$, so that
\begin{eqnarray}\label{eq:majmodk}
|\phi_{k,n}(\theta)|
  &\leq& 1-\sigma_{k,n}^2{\theta^2\over2}
         +\m{2+r}_{k,n}{|\theta|^{2+r}\over2}\nonumber\\
  &\leq& \exp(-\sigma_{k,n}^2
                   +\m{2+r}_{k,n}\delta_n^r){\theta^2\over2}.
\end{eqnarray}

In fact, (\ref{eq:majmodk}) also holds when $\sigma_{k,n}\delta_n\geq
1$, since in this case
\[
 -\sigma_{k,n}^2+\m{2+r}_{k,n}\delta_n^r
 \geq -\sigma_{k,n}^2+\sigma_{k,n}^{2+r}\delta_n^r
 \geq 0.
\]

 From~(\ref{eq:uan}), we can choose $n$ such that
$\sigma_{k,n}\leq\sigma_n/2$ and, using (\ref{eq:ineq}),
(\ref{eq:majdiffk}) and (\ref{eq:majmodk}), we find
\begin{eqnarray}\label{eq:majdiff}
\Bigl|\phi_n(\theta)-e^{-{\sigma_n^2\theta^2\over2}}\Bigr|
  &\leq&\sum_{k=1}^n\m{2+r}_{k,n}|\theta|^{2+r}
         \exp\Bigl(-\sigma^2_n+\sigma^2_{k,n}+\m{2+r}_n\delta_n^r
                  \Bigr){\theta^2\over2}\nonumber\\
  &\leq&\m{2+r}_n|\theta|^{2+r}
         \exp\Bigl(-\sigma^2_n{\theta^2\over8}\Bigr).
\end{eqnarray}

Equation~(\ref{eq:clt:norm}) follows, since the integral in (\ref{eq:In})
is bounded by
\begin{eqnarray*}
\int_{-\delta_n}^{\delta_n}
       \Bigl|\phi_n(\theta)-e^{-{\sigma_n^2\theta^2\over2}}\Bigr|d\theta
 &\leq& \m{2+r}_n\int_{-\infty}^\infty |\theta|^{2+r}
          \exp\Bigl(-\sigma^2_n{\theta^2\over8}\Bigr)d\theta\\
 &=& O\biggl({1\over\sigma_n}{\m{2+r}_n\over\sigma_n^{2+r}}\biggr).
\end{eqnarray*}

The proof of (\ref{eq:clt:norm2}) of the proposition is similar,
although the computations be more involved. Redefine $I_n$ as
\[
I_n \egaldef \int_{-\pi}^\pi e^{-\ci\theta x}\phi_n(\theta)d\theta
    -\int_{-\infty}^\infty e^{-\ci\theta x}
         \Bigl(1-\ci\bar\m3_n{\theta^3\over6}\Bigr)
            e^{-{\sigma_n^2\theta^2\over2}}d\theta,
\]

To find a bound for $|I_n|$, we have to estimate
\begin{eqnarray}\label{eq:diff:inteven}
\lefteqn{\biggl|\phi_n(\theta)-\Bigl(1-\ci\bar\m3_n{\theta^3\over6}\Bigr)
    e^{-{\sigma_n^2\theta^2\over2}}\biggr|}\qqqq\\
  &\leq& \biggl|\phi_n(\theta)
      -e^{-{\sigma_n^2\theta^2\over2}-\ci{\bar\m3_n\theta^3\over6}}\biggr|
      +\biggl|e^{-\ci{\bar\m3_n\theta^3\over6}}
              -1+\ci\bar\m3_n{\theta^3\over6}
       \biggr|e^{-{\sigma_n^2\theta^2\over2}}\nonumber.
\end{eqnarray}

The first part of the r.h.s.\ of (\ref{eq:diff:inteven}) is evaluated
as above with (\ref{eq:ineq}) and~(\ref{eq:majmodk}) replaced by
\[
\phi_{k,n}(\theta) 
  \leq \exp(-\sigma_{k,n}^2+\m3_{k,n}\delta_n){\theta^2\over2}.
\]

For the second part, we use the following inequality, valid for
$r\geq0$ (see e.g.\ Lo\`eve~\cite{Loe:1})
\[
 \biggl[{\m3_n\over \sigma_n^3}\biggr]^{1+{r\over3}}
    \leq{\m{3+r}_n\over\sigma_n^{3+r}},
\]
which yields
\[
 \biggl|e^{-\ci{\bar\m3_n\theta^3\over6}}
              -1+\ci\bar\m3_n{\theta^3\over6}\biggr|
 \leq \Bigl|\m3_n{\theta^3\over6}\Bigr|^{1+{r\over3}}
 \leq {\m{3+r}_n\over\sigma_n^{3+r}}{\sigma_n^{3+r}|\theta|^{3+r}\over6},
\]
and (\ref{eq:clt:norm2}) follows. 
\end{proof}

\begin{proof}{of Proposition~\ref{pro:lt:pole}}
The proof of this proposition is similar to the proof of
Proposition~\ref{pro:clt:norm} and is only sketched here. Define
\begin{eqnarray*}
 \omega(u) &\egaldef& {1\over1-\ci u},\\
 y_n &\egaldef& {\sum_{j\in\*F\0_n}m_{j,n} +x\over\alpha_n},
\end{eqnarray*}
and 
\begin{eqnarray}
I_n &\egaldef& 
  \alpha_n\int_{-\pi}^\pi 
     e^{-\ci\theta\alpha_ny_n}\omega_n^{\xi_n}(\theta)
       \prod_{k\in\*F\0_n}\psi_{k,n}(\theta)\wh\phi_n(\theta)d\theta\nonumber\\
  & & {}-\int_{-\infty}^\infty \!\!e^{-\ci uy_n}\omega^{\xi_n}(u)
      e^{-{\hat\sigma_n^2\over\alpha_n^2}{u^2\over2}}du
                                                                 \nonumber\\
  &=& \int_{-\pi\alpha_n}^{\pi\alpha_n}
        e^{-\ci uy_n}\omega_n^{\xi_n}(u/\alpha_n)
       \Bigl[\prod_{k\in\*F\0_n}\psi_{k,n}(u/\alpha_n)-1\Bigr]
           \wh\phi_n(u/\alpha_n)du                               \nonumber\\
  & & {}+\int_{-\pi\alpha_n}^{\pi\alpha_n}
        e^{-\ci uy_n}\omega_n^{\xi_n}(u/\alpha_n)\Bigl[\wh\phi_n(u/\alpha_n)
        -e^{-{\hat\sigma_n^2\over\alpha_n^2}{u^2\over2}}\Bigr]du
                                                                 \nonumber\\
  & & {}+\int_{-\pi\alpha_n}^{\pi\alpha_n}
        e^{-\ci uy_n}\Bigl[\omega_n^{\xi_n}(u/\alpha_n)
              -\omega^{\xi_n}(u)\Bigr]
        e^{-{\hat\sigma_n^2\over\alpha_n^2}{u^2\over2}}du
                                                                 \nonumber\\
  & & {}-\int_{|u|\geq\pi\alpha_n}e^{-\ci uy_n}\omega^{\xi_n}(u)
       e^{-{\hat\sigma_n^2\over\alpha_n^2}{u^2\over2}}du.
                                                         \label{eq:decompIn}
\end{eqnarray}

The evaluation of these integrals depends on the following
straightforward estimations, valid for $|u|<\pi\alpha_n$,
\begin{eqnarray*}
|\omega_n^{\xi_n}(u/\alpha_n)|
  &=& O\biggl({1\over(1+u^2)^{\xi_n/2}}\biggr),\\
|\omega_n^{\xi_n}(u/\alpha_n)-\omega^{\xi_n}(u)|
  &=& O\biggl({1\over\alpha_n}{u^2\over(1+u^2)^{\xi_n}}\biggr),\\
\Bigl|\prod_{k\in\*F\0_n}\psi_{k,n}(u/\alpha_n)-1\Bigr|
  &=& O\biggl({1+|u|\over\alpha_n}\biggr),
\end{eqnarray*}
and on~(\ref{eq:majdiff}), which yields for $|u|<\alpha_n\hat\delta_n$,
\[
 \Bigl|\wh\phi_n(u/\alpha_n)
   -e^{-{\hat\sigma_n^2\over\alpha_n^2}{u^2\over2}}\Bigr|
  = O\biggl({\hat\m{2+r}_n\over\alpha_n^{2+r}}\biggr)u^{2+r}
      \exp\Bigl(-{\hat\sigma_n^2\over\alpha_n^2}{u^2\over8}\Bigr),
\]
\[
 \Bigl|\wh\phi_n(u/\alpha_n)\Bigr|
  \leq \exp\Bigl(-{\hat\sigma_n^2\over\alpha_n^2}{u^2\over4}\Bigr).
\]
  
Moreover, we use the following approximation, valid for $a,b>0$ and
for sufficiently small $z$:
\[
 J(a,b,z)
  \egaldef\int_{-\infty}^\infty {|u|^a\over(1+u^2)^b}e^{-z^2u^2}du
  = O(1) + O(z^{2b-a-1}).
\]

These relations, together with~(\ref{eq:decompIn}), yield:
\begin{eqnarray*}
I_n 
  &=& O\biggl({1\over\alpha_n}\biggr)
           J\Bigl(1,{\xi_n\over2},{\hat\sigma_n\over2\alpha_n}\Bigr)
     +O\biggl({\hat\m{2+r}_n\over\alpha_n^{2+r}}\biggr)
       J\Bigl(2+r,{\xi_n\over2},{\hat\sigma_n\over\sqrt8\alpha_n}\Bigr)\\
  & & {}+O\biggl({1\over\alpha_n}\biggr)
           J\Bigl(2,\xi_n,{\hat\sigma_n\over\sqrt2\alpha_n}\Bigr)\\
  & & {}+O\biggl(\Bigl({\hat\sigma_n\over\alpha_n}\Bigr)^{\xi_n-1}\biggr)
               \int_{v\geq\pi\hat\sigma_n} 
                        v^{-\xi_n}e^{-{v^2\over2}}dv\\
  & & {}+O\biggl(\Bigl({\hat\gamma_n\over\alpha_n}\Bigr)^{\xi_n-1}\biggr)
                \int_{v\geq\delta_n\hat\gamma_n} 
                        v^{-\xi_n}e^{-{v^2\over2}}dv\\
  &=&O\biggl({1\over\alpha_n}+{\hat\m{2+r}_n\over\hat\sigma_n^{2+r}}
                  \Bigl({\hat\sigma_n\over\alpha_n}\Bigr)^{2+r}
                   +{\hat\m{2+r}_n\over\hat\sigma_n^{2+r}}
                  \Bigl({\hat\sigma_n\over\alpha_n}\Bigr)^{\xi_n-1}
                   \biggr)\\
  & &\mbox{}+O\biggl({e^{-{\hat\gamma_n^2\hat\delta_n^2\over5}}
               \over\hat\gamma_n^2\hat\delta_n^{\xi_n+1}\alpha_n^{\xi_n-1}}
        \biggr).
\end{eqnarray*}

To conclude the proof of (\ref{eq:pro:lt:pole}), the second term
coming in the definition of $I_n$ is evaluated using Parseval's
identity and classical tools of complex analysis (see e.g.\ Lavrentiev
and Chabat~\cite{LavCha:1}). This yields
\[
\int_{-\infty}^\infty e^{-\ci y_nu}\omega^{\xi_n}(u)
     e^{-{\hat\sigma_n^2\over\alpha_n^2}{u^2\over2}}du
 = {y_n^{\xi_n-1}e^{-y_n}\over\Gamma(\xi_n)}
    \biggl[1+O\biggl({\hat\sigma_n^2\over\alpha_n^2}\biggr)\biggr].
\]
\end{proof}

\end{secappendix}

\bibliography{queues}
\bibliographystyle{acm}

\end{document}